\documentclass[12pt]{article}

\usepackage[T1]{fontenc}
\usepackage[latin2]{inputenc}
\usepackage{lmodern}
\usepackage{amssymb}
\usepackage{stmaryrd}

\newtheorem{thm}{Theorem}

\begin{document}

\newcommand*{\R}{\ensuremath{\mathbb{R}}}
\newcommand*{\Om}{\ensuremath{\Omega}}
\newcommand*{\Oxy}{\ensuremath{\Omega(x,y)}}
\newcommand*{\poOxy}{\ensuremath{\partial_1\Omega(x,y)}}
\newcommand*{\ptOxy}{\ensuremath{\partial_2\Omega(x,y)}}
\newcommand*{\dtuk}{\ensuremath{\frac{\partial u_k}{\partial t}}}
\newcommand*{\dxiuk}{\ensuremath{\frac{\partial u_k}{\partial x_i}}}
\newcommand*{\dxiui}{\ensuremath{\frac{\partial u_i}{\partial x_i}}}
\newcommand*{\dxkuk}{\ensuremath{\frac{\partial u_k}{\partial x_k}}}
\newcommand*{\rtx}{\ensuremath{\rho (t,x)}}
\newcommand*{\dv}{\ensuremath{\mathrm{div}}}
\newcommand*{\grad}{\ensuremath{\mathrm{grad}}}
\newcommand*{\sqr}{\ensuremath{\square}}
\newcommand*{\rar}{\ensuremath{\ \Rightarrow\ }}
\newcommand*{\lar}{\ensuremath{\ \Leftarrow\ }}
\newcommand*{\lrar}{\ensuremath{\ \Leftrightarrow\ }}
\newcommand*{\lb}{\left(}
\newcommand*{\rb}{\right)}
\newcommand*{\Dom}{\ensuremath{\mathrm{Dom}}}
\newcommand*{\tni}{\ensuremath{\mathrm{int}}}

\title{Solution of a Simple Case of the Navier-Stokes Equations via Employing the
Lambert W Function}
\author{J\'ozsef Vass\\E\"otv\"os Lor\'and University of Sciences, Budapest,
Hungary\\vass@cs.elte.hu}
\date{11-26-2003}
\maketitle

\begin{abstract}
The purpose of this paper is to introduce a curious function of two variables,
expressable via the employment of the Lambert W Funtion, which can be
generalized to satisfy Euler's
Equation of Inviscid Motion over a specific domain, with pressure independent of space
variables.
\end{abstract}

The significance of the Navier-Stokes equations is well known, and the
questions related to it, which seem to be unanswerable with our current knowledge of
mathematics. For more details see \cite{it_4}.

I am convinced that all ideas, however
small or seemingly insignificant, may contribute to some first important steps in
answering those questions. For this reason, I considered it important to publish
the simple results contained in this paper.

In recent years,
researchers in various applied and pure fields of mathematics have begun to take
notice of a function called the Lambert W Function. Many mathematicians, including me,
are convinced, that this function should have its proper place among the basic
elementary functions, and it is indeed likely to take that place in the 21st century.
The basis of this view is provided by the ever increasing number of reported
applications of this function from fields such as combinatorics, analysis of
algorithms, asymptotic roots of trinomials, jet fuel problems,
combustion problems, enzyme kinetics, explanation of an anomaly in the calculation of
exchange forces, similarity solution for the Richards
Equation, Volterra Equations for population growth,
epidemics and components. Many of these applications were
found simply because the Lambert W function has fortunately been implemented in Maple.
For more details see \cite{it_1, it_2, it_3}.\\
\\
Let us denote
\[ \Om(x,y):= \frac{y}{x}-W\lb-\frac{1}{x}e^{\frac{y}{x}}\rb \]
where $W$ denotes the principal branch of the Lambert W Function \cite{it_2}, meaning
\[ w=W(x) \lrar x=we^w, w\in\R, w\ge -1 \]
The domain of \Om\ depends on the domain of $W$, so we have the following proposition.\\

\newpage
\noindent
\textbf{Proposition 1}
\[ \Dom(\Om)=\left\{ (x,y)\in\R^2:\ x<0\ \mathrm{or}\ \left( x>0\ \mathrm{and}\ y\le
x\ln\frac{x}{e}\right)\right\} \]
\\
\textbf{Proof}\ \ $\Dom(W)=\{x\in\R:\ x\ge -\frac{1}{e}\}$, so we need
\[ -\frac{1}{x}e^{\frac{y}{x}}\ge-\frac{1}{e} \lrar
\frac{1}{x}e^{\frac{y}{x}}\le\frac{1}{e} \]
Case of $x<0:\ e^{\frac{y}{x}}\ge\frac{1}{e}x$, which is true for all $y$.\\
Case of $x>0$:
\[ e^{\frac{y}{x}}\le\frac{1}{e}x \lrar \frac{y}{x}\le \ln\frac{x}{e} \lrar
y\le x\ln\frac{x}{e}\ \ \sqr \]
\\
The following lemma shows that $\Om$ describes the places of intersection
of linear functions and the exponential function.\\
\\
\textbf{Lemma}
\[ e^{\Oxy}=x\ \Oxy -y\ \ ((x,y)\in \Dom(\Om)) \]
\[ \partial_1\Oxy + \Oxy\cdot\partial_2\Oxy = 0\ \ ((x,y)\in \tni\ \Dom(\Om)) \]
\textbf{Proof}\\
Property 1:
\[-\frac{1}{x}e^{\frac{y}{x}}=W\lb
-\frac{1}{x}e^{\frac{y}{x}}\rb\exp\lb W\lb-\frac{1}{x}e^{\frac{y}{x}}\rb\rb
\rar\]
\[\exp\lb\frac{y}{x}-W\lb-\frac{1}{x}e^{\frac{y}{x}}\rb\rb=-x
W\lb-\frac{1}{x}e^{\frac{y}{x}}\rb=x\lb\frac{y}{x}-W\lb-\frac{1}{x}e^{\frac{y}{x}}\rb\rb-y
\]
Property 2:
\[e^{\Oxy} \partial_1 \Oxy = \Oxy + x\poOxy\]
\[e^{\Oxy} \ptOxy = x\ptOxy -1\]
\[ \rar (\poOxy + \Oxy\cdot\ptOxy)(e^{\Oxy} -x) = 0 \rar\]
\[ (a)\ \Oxy = \ln x\ \ \mathrm{or}\ \ (b)\ \poOxy + \Oxy\cdot\ptOxy = 0 \]
\[ (a) \lrar \frac{y}{x}-W\lb-\frac{1}{x}e^{\frac{y}{x}}\rb=\ln x \lrar
e^{\frac{y}{x}}=xe^{W(-\frac{1}{x}e^{\frac{y}{x}})} \lrar \]
\[ W\lb-\frac{1}{x}e^{\frac{y}{x}}\rb e^{W(-\frac{1}{x}e^{\frac{y}{x}})} =
-\frac{1}{x}e^{\frac{y}{x}} = (-1)e^{W(-\frac{1}{x}e^{\frac{y}{x}})} \lrar\]
\[ W\lb-\frac{1}{x}e^{\frac{y}{x}}\rb = -1 \lrar
-\frac{1}{x}e^{\frac{y}{x}}=-\frac{1}{e} \lrar
 y=x\ln\frac{x}{e} \rar (x,y)\notin \tni\ \Dom(\Om) \]
\sqr\\
\\
We now investigate where some special values of the $\Om$ function occur.\\
\\
\textbf{Proposition 2}\\
\[ \Oxy=0\ \lrar\ y=-1\ \ ((x,y)\in \Dom(\Om)) \]
\[ \Oxy=\ln x\ \lrar\ y=x\ln\frac{x}{e}\ \ ((x,y)\in \Dom(\Om)) \]
\[ \Oxy = C + \ln y\ \lrar\ y =
-\frac{x}{1+e^C}W\lb-(1+e^{-C})\frac{1}{x}\rb\ \ ((x,y)\in \Dom(\Om),\ C\in\R) \]
\textbf{Proof}\\
Property 1:
\[ \frac{y}{x}=W(-\frac{1}{x}e^{\frac{y}{x}}) \lrar
\frac{y}{x}e^{\frac{y}{x}}=W^{-1}(\frac{y}{x})=-\frac{1}{x}e^{\frac{y}{x}} \lrar
y=-1 \]
Property 2: See in the Lemma, in the proof of Property 2.\\
Property 3:
\[ C+\ln y = \frac{y}{x}-W\lb-\frac{1}{x}e^{\frac{y}{x}}\rb \lrar
-\frac{1}{x}e^{\frac{y}{x}}=\lb\frac{y}{x}-C-\ln y\rb\exp\lb\frac{y}{x}-C-\ln y\rb
\lrar \]
\[ -y = \frac{1}{e^C}(y-Cx-x\ln y) \lrar \lb 1+\frac{1}{e^C}\rb y =
\frac{x}{e^C}(C+\ln y)
\lrar \]
\[ \exp\lb (e^C +1)\frac{y}{x}\rb = e^C y \lrar -\frac{1+e^C}{e^C x} =
-(1+e^C)\frac{y}{x}\exp\lb -(1+e^C)\frac{y}{x}\rb \lrar \]
\[ W\lb -(1+e^{-C})\frac{1}{x}\rb = -(1+e^C)\frac{y}{x} \lrar y =
-\frac{x}{1+e^C}W\lb -(1+e^{-C})\frac{1}{x}\rb\ \sqr\]
\\
Let us now look at limits of $\Om$ at zero taken from the left and right (denoted by
$x\to 0^{\pm}$), as well as in infinity.\\
\\
\textbf{Proposition 3}
\[ \lim_{x\to 0^-} \Oxy = -\infty\ (y\ge 0),\ \ \lim_{x\to 0^-} \Oxy = \ln(-y)\ (y < 0) \]
\[ \lim_{x\to 0^+} \Oxy = -\infty\ (y < 0),\ \ \lim_{x\to\pm\infty} \Oxy= 0\ (y\in\R) \]
\\
\textbf{Proof}\ \ Let $y$ be any positive number.
\[ \lim_{x\to 0^-} -\frac{1}{x}e^{\frac{y}{x}} = \lim_{x\to +\infty} \frac{x}{e^{yx}}
= 0 \rar \lim_{x\to 0^-} \Oxy = -\infty -0 \]
In the case of $y=0$ we have
\[ \lim_{x\to 0^-} \Om(x,0) = \lim_{x\to 0^-} -W\lb -\frac{1}{x}\rb = -\infty \]
Now, let $y$ denote any negative number.
\[ \lim_{x\to 0^-} -\frac{1}{x}e^{\frac{y}{x}} = +\infty \rar
 W\lb-\frac{1}{x}e^{\frac{y}{x}}\rb = \ln\lb -\frac{1}{x}e^{\frac{y}{x}}\rb - \ln
W\lb-\frac{1}{x}e^{\frac{y}{x}}\rb \lrar \]
\[ \Oxy =
\ln\lb-xW\lb-\frac{1}{x}e^{\frac{y}{x}}\rb\rb \]
if $|x|$ is small enough. So
\[ \lim_{x\to 0^-} \Oxy = \lim_{x\to 0^-}
\ln\lb-xW\lb-\frac{1}{x}e^{\frac{y}{x}}\rb\rb \]
\[ \lim_{x\to 0^-} -x W\lb-\frac{1}{x} e^{\frac{y}{x}}\rb = \lim_{x\to +\infty}
\frac{W(xe^{-yx})}{x} = \]
\[ =\lim_{x\to +\infty}
\frac{W(xe^{-yx})}{xe^{-yx}(1+W(xe^{-yx}))}e^{-yx}(1-yx) = -y \]
by l'Hopital's Rule.
\[ \lim_{x\to 0^+} -\frac{1}{x}e^{\frac{y}{x}} = \lim_{x\to +\infty} -xe^{yx} = 0 \rar
 \lim_{x\to 0^+} \Oxy = \lim_{x\to 0^+} \frac{y}{x} - W(0) = -\infty \]
Now, let $y$ be any real number.
\[ \lim_{x\to\pm\infty} \Oxy = 0 - W(0) = 0 \]
\sqr\\
\\
We now turn to the main result of this paper. Let us denote
\[ u_k(t,x):=\Om(t,x_k)\ \ (k=1,\dots,n),\ \ u:=(u_1,\dots,u_n) \]
where
\[ \Dom(u):=\bigcap_{j=1}^n \Dom(u_j) \]
where
\[ \Dom(u_j):=\left\{ (t,x)\in\R^{n+1}:\ t<0\ \mathrm{or}\ \lb t>0\ \mathrm{and}\
x_j\le t\ln\frac{t}{e}\rb\right\} \]

\begin{thm}
\[ \dtuk + \sum_{i=1}^n u_i \dxiuk = 0\ \ (\mathrm{on}\ \tni\ \Dom(u),
\ k=1,\dots,n) \]
\end{thm}
\textbf{Proof}
\[ \dtuk(t,x) + \sum_{i=1}^n u_i(t,x)\dxiuk(t,x) = \dtuk(t,x) + u_k(t,x)\dxkuk(t,x) = \]
\[ = \partial_1\Om(t,x_k) + \Om(t,x_k)\partial_2\Om(t,x_k) = 0 \]
by our Lemma. \sqr\\
\\
This result is equivalent to stating that our $u=(u_1,\dots,u_n)$ mapping satisfies
Euler's Equation of Inviscid Motion over a specified domain, with the pressure
function independent of space variables, which is really just a special case of the
Navier-Stokes Equations. From the derivatives of the $\Om$ function, we see
that our mapping $u$ will not satisfy the incompressibility requirement
\[ \mathrm{div}\ u = \sum_{i=1}^n \dxiui = 0 \]
generally included in the Navier-Stokes Equations. I feel that I must also state, I am
not entirely convinced that this requirement is really a basic characteristic of
fluids in nature. We may also investigate however, the interesting question of whether
we can find a function $\rho$, for which the law of mass conservation would
hold, with our mapping $u$ defined above. Let us denote
\[ \rtx := \prod_{k=1}^n \frac{1}{e^{\Om (t,x_k)}-t},\ \ \Dom(\rho):=\tni\ \Dom(u) \]

\begin{thm}
\[ \frac{\partial \rho}{\partial t} + \dv(\rho u) = 0\ \ (\mathrm{on}\ \tni\ \Dom(u)) \]
\end{thm}
\textbf{Proof}
\[ \frac{\partial\rho}{\partial t} (t,x) = \rtx\sum_{i=1}^n
\frac{-1}{e^{\Om(t,x_i)}-t} (\partial_1 \Om (t,x_i) e^{\Om(t,x_i)} -1) \]
\[ \frac{\partial\rho}{\partial x_k} (t,x) = \rho(t,x)\frac{-1}{e^{\Om(t,x_k)}-t}
\partial_2 \Om (t,x_k) e^{\Om(t,x_k)} \]
\[ \lb \frac{\partial \rho}{\partial t} + \langle u,\grad\ \rho\rangle \rb (t,x) = \]
\[ = \rtx\sum_{i=1}^n \frac{1}{e^{\Om(t,x_i)}-t} + \rtx\sum_{i=1}^n
\frac{-1}{e^{\Om(t,x_i)}-t} e^{\Om(t,x_i)}
(\partial_1\Om(t,x_i)+\Om(t,x_i)\partial_2\Om(t,x_i)) = \]
\[ =(-\rho \dv\ u)(t,x) \]
\sqr

\vspace{2cm}

\noindent
I would like to express my gratitude to Prof. Gisbert Stoyan and Istv\'an Sigray for their
helpful suggestions and comments regarding this paper.

\end{document}